\documentclass[12pt]{article}
\usepackage{amsmath}
\usepackage{amssymb}
\usepackage{latexsym}
\usepackage{amsthm}
\usepackage{mathrsfs}
\usepackage{enumitem}
\usepackage[colorlinks,
            linkcolor=red,
            anchorcolor=blue,
            citecolor=green
            ]{hyperref}
\usepackage{graphicx}
\usepackage{diagbox}
\usepackage{tikz}

\newtheorem{thm}{Theorem}[section]
\newtheorem{lem}[thm]{Lemma}
\newtheorem{prop}[thm]{Proposition}

\newtheorem{definition}{Definition}

\makeatother

\begin{document}

\begin{center}
{\large \bf  Pattern-avoiding alternating words}
\end{center}

\begin{center}
Emma L.L. Gao$^{1}$,
Sergey Kitaev$^{2}$, and Philip B. Zhang$^{3}$\\[6pt]

$^{1,3}$Center for Combinatorics, LPMC-TJKLC,\\
Nankai University, Tianjin 300071, P. R. China\\[6pt]

$^{2}$ Department of Computer and Information Sciences, \\
University of Strathclyde, 26 Richmond Street, Glasgow G1 1XH, UK\\[6pt]

Email: $^{1}${\tt gaolulublue@mail.nankai.edu.cn},
	   $^{2}${\tt sergey.kitaev@cis.strath.ac.uk},
           $^{3}${\tt zhangbiaonk@163.com}
\end{center}

\noindent\textbf{Abstract.}
A word $w=w_1w_2\cdots w_n$ is alternating if either $w_1<w_2>w_3<w_4>\cdots$ (when the word is up-down) or $w_1>w_2<w_3>w_4<\cdots$ (when the word is down-up). In this paper, we initiate the study of (pattern-avoiding) alternating words. We enumerate up-down (equivalently, down-up) words via finding a bijection with order ideals of a certain poset. Further, we show that the number of 123-avoiding up-down words of even length is given by the Narayana numbers, which is also the case, shown by us bijectively, with 132-avoiding up-down words of even length. 
We also give formulas for enumerating  all other cases of avoidance of a permutation pattern of length 3 on alternating words. 



\section{Introduction}\label{intro}
A permutation $\pi=\pi_1\pi_2\cdots \pi_n$ is called {\em up-down} if $\pi_1<\pi_2>\pi_3<\pi_4>\pi_5<\cdots$.  A permutation $\pi=\pi_1\pi_2\cdots \pi_n$ is called {\em down-up} if $\pi_1>\pi_2<\pi_3>\pi_4<\pi_5>\cdots$. A famous result of Andr\'{e} is saying that if $E_n$ is the number of up-down (equivalently, down-up) permutations of $1,2,\ldots,n$, then $$\sum_{n\geq 0}E_n\frac{x^n}{n!}=\sec x+\tan x.$$ Some aspects of up-down and down-up permutations, also called {\em reverse alternating} and {\em alternating}, respectively, are surveyed in~\cite{Stanley2010survey}. Slightly abusing these definitions, we refer to {\em alternating permutations} as the union of up-down and down-up permutations. This union is known as the set of {\em zigzag permutations}.

In this paper, we extend the study of alternating permutations to that of {\em alternating words}. These words, also called {\em zigzag words}, are the union of up-down and down-up words, which are defined in a similar way to the definition of up-down and down-up permutations, respectively. For example, $1214$, $2413$, $2424$ and $3434$ are examples of up-down words of length 4 over the alphabet $\{1,2,3,4\}$.

Section~\ref{en-alt-words} is dedicated to the enumeration of up-down words, which is equivalent to enumerating down-up words by applying the operation of {\em complement}. For a word $w=w_1w_2\cdots w_n$ over the alphabet $\{1,2,\ldots,k\}$ its complement $w^c$ is the word $c_1c_2\cdots c_n$, where for each $i=1,2,\ldots,n$, $c_i=k+1-w_i$. For example, the complement of the word $24265$ over the alphabet $\{1,2,\ldots,6\}$ is $53512$. Our enumeration in Section~\ref{en-alt-words} is done by linking bijectively up-down words to order ideals of certain posets and using known results.

A ({\em permutation}) {\em pattern} is a permutation $\tau=\tau_1\tau_2\cdots\tau_k$. We say that a permutation $\pi=\pi_1\pi_2\cdots\pi_n$ {\em contains an occurrence} of $\tau$ if there are $1\leq i_1< i_2<\cdots< i_k\leq n$ such that $\pi_{i_1}\pi_{i_2}\cdots \pi_{i_k}$ is order-isomorphic to $\tau$. If $\pi$ does not contain an occurrence of $\tau$, we say that $\pi$ {\em avoids}~$\tau$. For example, the permutation 315267 contains several occurrences of the pattern 123, for example, the subsequences 356 and 157, while this permutation avoids the pattern 321. Occurrences of a pattern in words are defined similarly as subsequences order-isomorphic to a given word called pattern (the only difference with permutation patterns is that word patterns can contain repetitive letters, which is not in the scope of this paper).

A comprehensive introduction to the theory of patterns in permutations and words can be found in~\cite{Kitaev2011Patterns}. In particular, Section 6.1.8 in~\cite{Kitaev2011Patterns} discusses known results on pattern-avoiding alternating permutations, and Section 7.1.6 discusses results on permutations avoiding patterns in a more general sense.

In this paper we initiate the study of pattern-avoiding alternating words. In Section~\ref{123-up-down-sec} we enumerate up-down words over $k$-letter alphabet avoiding the pattern 123. In particular, we show that in the case of even length, the answer is given by the {\em Narayana numbers} counting, for example, {\em Dyck paths} with a specified number of peaks (see Theorem~\ref{main-thm}). Interestingly, the number of 132-avoiding words over $k$-letter alphabet of even length is also given by the Narayana numbers, which we establish bijectively in Section~\ref{bijection-sec}. In Section~\ref{sec-132-av} we provide a (non-closed form) formula for the number of 132-avoiding words over $k$-letter alphabet of odd length. In Section~\ref{312-up-down-sec} we show that the enumeration of 312-avoiding up-down words is equivalent to that of 123-avoiding up-down words. Further, a classification of all cases of avoiding a length 3 permutation pattern on up-down words is discussed in Section~\ref{all-cases}. Finally, some concluding remarks are given in Section~\ref{last-sec}.

In what follows, $[k]=\{1,2,\ldots,k\}$.

\section{Enumeration of up-down words}\label{en-alt-words}

In this section, we consider the enumeration of up-down words.
We shall show that this problem is the same as that of  enumerating order ideals of a certain poset.
Since up-down words are in one-to-one-correspondence with down-up words by using the complement operation, we consider only down-up words throughout this section.

Table~\ref{tab1} provides the number $N_{k,\ell}$ of down-up words of length $\ell$ over the alphabet $[k]$ for small values of $k$ and $\ell$ indicating connections to the {\em Online Encyclopedia of Integer Sequences} ({\em OEIS})~\cite{SloaneLine}.

\begin{table}[!htb]
\small
\begin{center}
\begin{tabular}{|c|lllllllllll|l|}
\hline
\diagbox{$k$}{$\ell$}& 0 & 1 & 2 & 3 & 4 &5 &6 & 7 &8&9 &10& OEIS\\
\hline
2&  1& 2&1& 1&1& 1&1& 1&1& 1&1& trivial \\
\hline
3&  1& 3& 3& 5& 8& 13& 21& 34& 55& 89& 144 & A000045\\
\hline
4&  1&  4& 6& 14& 31& 70& 157& 353& 793& 1782& 4004 & A006356\\
\hline
5&  1&  5& 10& 30& 85& 246& 707& 2037& 5864& 16886& 48620& A006357\\
\hline
6&  1&  6& 15& 55& 190& 671& 2353& 8272& 29056 &  102091 & 358671 & A006358\\
\hline
7&  1&  7& 21& 91& 371& 1547& 6405& 26585& 110254 & 457379 & 1897214 & A006359\\
\hline
\end{tabular}
\end{center}
\caption{The number $N_{k,\ell}$ of down-up words on $[k]$ of length $\ell$ for small values of $k$ and~$\ell$.}\label{tab1}
\end{table}

We assume the reader is familiar with the notion of a partially ordered set (poset) and some basic properties of posets; e.g. see \cite{Stanley1997Enumerative}. A partially ordered set $P$ is a set  together with a binary relation denoted by $\leq_P$ that satisfies the properties of reflexivity, antisymmetry and transitivity. An order ideal of $P$ is a subset $I$ of $P$ such that if $x\in I$ and $y\leq x$ then $y\in I$. We denote $J(P)$ the set  of all order ideals of $P$.

Let $\mathbf{n}$ be the poset on $[n]$ with its usual order ($\mathbf{n}$ is a linearly ordered set).
The {\em $m$-element zigzag poset}, denoted $Z_m$, is shown schematically in Figure \ref{Zm}. Note that the order $<_{Z_m}$ in $Z_m$ is $1<2>3<4>5<\cdots.$ The definition of the order $\leq_{Z_m}$ is self-explanatory.

\begin{figure}[!htb]
\small
    \centering
    \begin{minipage}{.5\textwidth}
        \centering
        \begin{tikzpicture}[scale=1.4]
	\draw (3,0)--(3.5,1)--(4,0)--(4.5,1)--(5,0);
	\fill(3,0) circle(0.04cm);    \node [below] at (3,0) {$1$};
	\fill(3.5,1) circle(0.04cm); \node [above] at (3.5,1) {$2$};
	\fill(4,0) circle(0.04cm);    \node [below] at (4,0) {$3$};
	\fill(4.5,1) circle(0.04cm); \node [above] at (4.5,1) {$4$};
	\fill(5,0) circle(0.04cm);    \node [below] at (5,0) {$5$};
	
	\fill(5.2,0.5) circle(0.02cm); \fill(5.5,0.5) circle(0.02cm); \fill(5.8,0.5) circle(0.02cm);

	\draw (6,0)--(6.5,1);
	\fill(6,0) circle(0.04cm);    \node [below]  at (6,0){$m-1$};
	\fill(6.5,1) circle(0.04cm); \node [above]  at (6.5,1){$m$};
	
	\coordinate[label=below:$m$ even] (1) at (5,-0.5);
        \end{tikzpicture}
    \end{minipage}%
    \begin{minipage}{0.5\textwidth}
        \centering
        \begin{tikzpicture}[scale=1.4]
	\draw (3,0)--(3.5,1)--(4,0)--(4.5,1)--(5,0);

	\fill(3,0) circle(0.04cm);    \node [below] at (3,0) {$1$};
	\fill(3.5,1) circle(0.04cm); \node [above] at (3.5,1) {$2$};
	\fill(4,0) circle(0.04cm);    \node [below] at (4,0) {$3$};
	\fill(4.5,1) circle(0.04cm); \node [above] at (4.5,1) {$4$};
	\fill(5,0) circle(0.04cm);    \node [below] at (5,0) {$5$};
	
	\fill(5.2,0.5) circle(0.02cm); \fill(5.5,0.5) circle(0.02cm); \fill(5.8,0.5) circle(0.02cm);

	\draw (6,1)--(6.5,0);
	\fill(6,1) circle(0.04cm);    \node [above]  at (6,1){\small $m-1$};
	\fill(6.5,0) circle(0.04cm); \node [below]  at (6.5,0){\small $m$};
	\coordinate[label=below:$m$ odd] (1) at (5,-0.5);
	\end{tikzpicture}
    \end{minipage}
  \caption{The zigzag poset $Z_{m}$.}
  \label{Zm}
\end{figure}
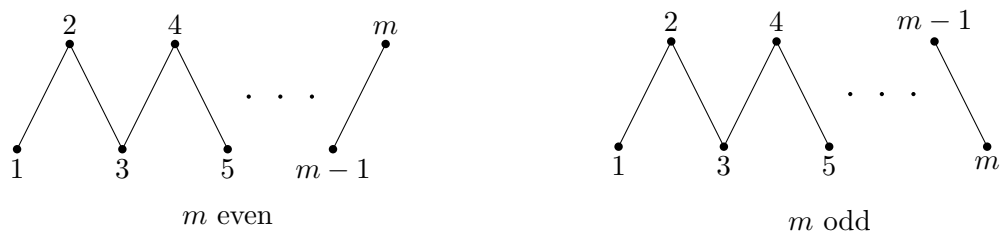

The poset $Z_m \times \mathbf{n}$ is as shown in Figure~\ref{Zmn}. Elements of $Z_m \times \mathbf{n}$ are pairs $(i,j)$, where $i \in Z_m$ and $j\in [n]$, and the order is defined as follows:
$$(i,j)\leq (k,\ell) \mbox{ if and only if } i\leq_{Z_{m}}k\mbox{ and } j\leq \ell.$$

\begin{figure}[!htb]
\tiny
    \centering
    \begin{minipage}{.5\textwidth}
        \centering
        \begin{tikzpicture}[scale=1.5]
      \draw (3,1)--(3.5,2)--(4,1)--(4.5,2)--(5,1);
      \draw (3,-1)--(3.5,0)--(4,-1)--(4.5,0)--(5,-1);
      \draw (3,-2)--(3.5,-1)--(4,-2)--(4.5,-1)--(5,-2);
      \draw (3,1)--(3,0.7); \draw  (3,-0.3)--(3,-1)--(3,-2);
      \draw (3.5,2)--(3.5,1.3); \draw  (3.5,0.3)--(3.5,0)--(3.5,-1);
      \draw (4,1)--(4,0.7); \draw  (4,-0.3)--(4,-1)--(4,-2);
      \draw (4.5,2)--(4.5,1.3); \draw  (4.5,0.3)--(4.5,0)--(4.5,-1);
      \draw (5,1)--(5,0.7); \draw  (5,-0.3)--(5,-1)--(5,-2);
      \fill(3,1) circle(0.04cm); \node [right] at (3,1){  $(1,n)$};
      \fill(3.5,2) circle(0.04cm); \node [right] at  (3.5,2){  $(2,n)$};
      \fill(4,1) circle(0.04cm);  \node [right] at (4,1){  $(3,n)$};
      \fill(4.5,2) circle(0.04cm);   \node [right] at (4.5,2){  $(4,n)$};
      \fill(5,1) circle(0.04cm);  \node [right] at (5,1) {  $(5,n)$};
      \fill(3,-1) circle(0.04cm); \node [right] at (3,-1) {  $(1,2)$};
      \fill(3.5,0) circle(0.04cm); \node [right] at (3.5,0) {  $(2,2)$};
      \fill(4,-1) circle(0.04cm); \node [right] at (4,-1) {  $(3,2)$};
      \fill(4.5,0) circle(0.04cm);  \node [right] at (4.5,0) {  $(4,2)$};
      \fill(5,-1) circle(0.04cm); \node [right] at (5,-1) {  $(5,2)$};
      \fill(3,-2) circle(0.04cm);  \node [right] at (3,-2) {  $(1,1)$};
      \fill(3.5,-1) circle(0.04cm); \node [right] at (3.5,-1) {  $(2,1)$};
      \fill(4,-2) circle(0.04cm); \node [right] at (4,-2) {  $(3,1)$};
      \fill(4.5,-1) circle(0.04cm);\node [right] at (4.5,-1) {  $(4,1)$};
      \fill(5,-2) circle(0.04cm); \node [right] at(5,-2) {  $(5,1)$};

      \fill(3.7,0.5) circle(0.02cm); \fill(4,0.5) circle(0.02cm); \fill(4.3,0.5) circle(0.02cm);
      \fill(5.2,1.5) circle(0.02cm); \fill(5.5,1.5) circle(0.02cm); \fill(5.8,1.5) circle(0.02cm);
      \fill(5.2,0.5) circle(0.02cm); \fill(5.5,0.5) circle(0.02cm); \fill(5.8,0.5) circle(0.02cm);
      \fill(5.2,-0.5) circle(0.02cm); \fill(5.5,-0.5) circle(0.02cm); \fill(5.8,-0.5) circle(0.02cm);
      \fill(5.2,-1.5) circle(0.02cm); \fill(5.5,-1.5) circle(0.02cm); \fill(5.8,-1.5) circle(0.02cm);

      \draw (6,1)--(6.5,2);\draw (6,-1)--(6.5,0);\draw (6,-2)--(6.5,-1);
      \draw (6,1)--(6,0.7); \draw  (6,-0.3)--(6,-1)--(6,-2);
      \draw (6.5,2)--(6.5,1.3); \draw  (6.5,0.3)--(6.5,0)--(6.5,-1);
      \fill(6,1) circle(0.04cm);  \node [right] at (6,1) {  $(m-1,n)$};
      \fill(6.5,2) circle(0.04cm); \node [right] at (6.5,2) {   $(m,n)$};
      \fill(6,-1) circle(0.04cm);  
      \fill(6.5,0) circle(0.04cm); \node [right] at (6.5,0) {   $(m,2)$};
      \fill(6,-2) circle(0.04cm); \node [right] at (6,-2) {   $(m-1,1)$};
      \fill(6.5,-1) circle(0.04cm); \node [right] at (6.5,-1) {   $(m,1)$};
    \coordinate[label=below:\small $m$ even] (1) at (5,-2.5);
        \end{tikzpicture}
    \end{minipage}%
    \begin{minipage}{0.5\textwidth}
        \centering
        \begin{tikzpicture}[scale=1.5]
     \draw (3,1)--(3.5,2)--(4,1)--(4.5,2)--(5,1);
      \draw (3,-1)--(3.5,0)--(4,-1)--(4.5,0)--(5,-1);
      \draw (3,-2)--(3.5,-1)--(4,-2)--(4.5,-1)--(5,-2);
      \draw (3,1)--(3,0.7); \draw  (3,-0.3)--(3,-1)--(3,-2);
      \draw (3.5,2)--(3.5,1.3); \draw  (3.5,0.3)--(3.5,0)--(3.5,-1);
      \draw (4,1)--(4,0.7); \draw  (4,-0.3)--(4,-1)--(4,-2);
      \draw (4.5,2)--(4.5,1.3); \draw  (4.5,0.3)--(4.5,0)--(4.5,-1);
      \draw (5,1)--(5,0.7); \draw  (5,-0.3)--(5,-1)--(5,-2);
      \fill(3,1) circle(0.04cm); \node [right] at (3,1){  $(1,n)$};
      \fill(3.5,2) circle(0.04cm); \node [right] at  (3.5,2){  $(2,n)$};
      \fill(4,1) circle(0.04cm);  \node [right] at (4,1){  $(3,n)$};
      \fill(4.5,2) circle(0.04cm);   \node [right] at (4.5,2){  $(4,n)$};
      \fill(5,1) circle(0.04cm);  \node [right] at (5,1) {  $(5,n)$};
      \fill(3,-1) circle(0.04cm); \node [right] at (3,-1) {  $(1,2)$};
      \fill(3.5,0) circle(0.04cm); \node [right] at (3.5,0) {  $(2,2)$};
      \fill(4,-1) circle(0.04cm); \node [right] at (4,-1) {  $(3,2)$};
      \fill(4.5,0) circle(0.04cm);  \node [right] at (4.5,0) {  $(4,2)$};
      \fill(5,-1) circle(0.04cm); \node [right] at (5,-1) {  $(5,2)$};
      \fill(3,-2) circle(0.04cm);  \node [right] at (3,-2) {  $(1,1)$};
      \fill(3.5,-1) circle(0.04cm); \node [right] at (3.5,-1) {  $(2,1)$};
      \fill(4,-2) circle(0.04cm); \node [right] at (4,-2) {  $(3,1)$};
      \fill(4.5,-1) circle(0.04cm);\node [right] at (4.5,-1) {  $(4,1)$};
      \fill(5,-2) circle(0.04cm); \node [right] at(5,-2) {  $(5,1)$};

      \fill(3.7,0.5) circle(0.02cm); \fill(4,0.5) circle(0.02cm); \fill(4.3,0.5) circle(0.02cm);
      \fill(5.2,1.5) circle(0.02cm); \fill(5.5,1.5) circle(0.02cm); \fill(5.8,1.5) circle(0.02cm);
      \fill(5.2,0.5) circle(0.02cm); \fill(5.5,0.5) circle(0.02cm); \fill(5.8,0.5) circle(0.02cm);
      \fill(5.2,-0.5) circle(0.02cm); \fill(5.5,-0.5) circle(0.02cm); \fill(5.8,-0.5) circle(0.02cm);
      \fill(5.2,-1.5) circle(0.02cm); \fill(5.5,-1.5) circle(0.02cm); \fill(5.8,-1.5) circle(0.02cm);

      \draw (6,2)--(6.5,1);\draw (6,0)--(6.5,-1);\draw (6,-1)--(6.5,-2);
      \draw (6,2)--(6,1.3); \draw  (6,0.3)--(6,0)--(6,-1);
      \draw (6.5,1)--(6.5,0.7); \draw  (6.5,-0.3)--(6.5,-1)--(6.5,-2);
      \fill(6,2) circle(0.04cm);  \node [right] at (6,2) {  $(m-1,n)$};
      \fill(6.5,1) circle(0.04cm); \node [right] at (6.5,1) {   $(m,n)$};
      \fill(6,0) circle(0.04cm);  \node [right] at (6,0) {   $(m-1,2)$};
      \fill(6.5,-1) circle(0.04cm); \node [right] at (6.5,-1) {   $(m,2)$};
      \fill(6,-1) circle(0.04cm); 
      \fill(6.5,-2) circle(0.04cm); \node [right] at (6.5,-2) {   $(m,1)$};
    \coordinate[label=below:\small $m$ odd] (1) at (5,-2.5);
        \end{tikzpicture}
    \end{minipage}
  \caption{The poset $Z_{m} \times \mathbf{n}$.}
  \label{Zmn}
\end{figure}
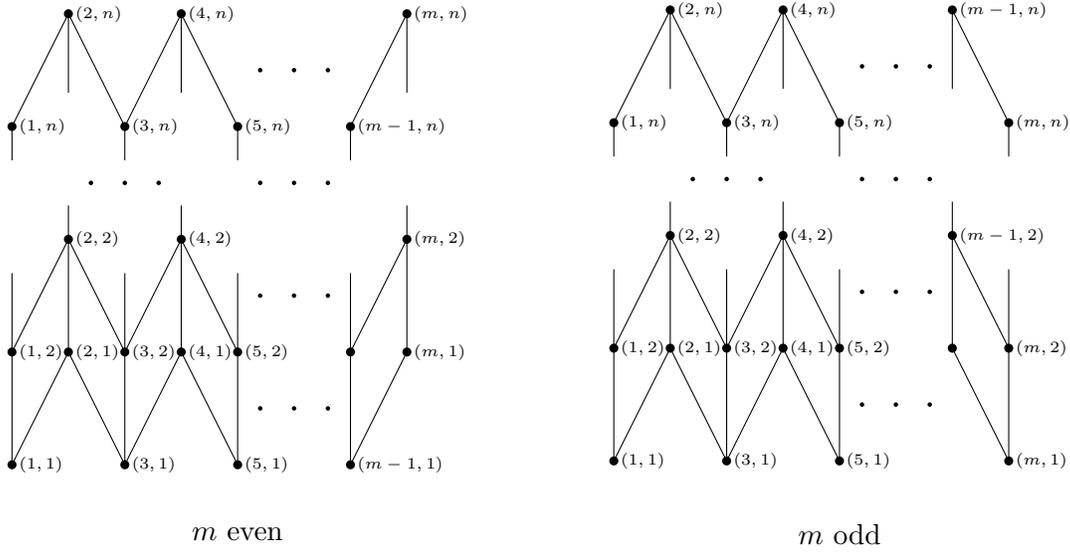

It is known that, for $m\ge 2$, the size of  $J(Z_m)$ equals to the Fibonacci number $F_{m+2}$, which is defined recursively as $F_1=F_2=1$ and $F_{n+1}=F_n+F_{n-1}$ for  any $n\geq 2$; see Stanley \cite[Ch. 3 Ex. 23.a]{Stanley1997Enumerative}.
The enumeration of $J(Z_m \times \mathbf{n})$ was studied by Berman and K\"ohler \cite{Berman1976Cardinalities}.
The following theorem reveals their connection with the enumeration of alternating words. We shall give two proofs of it here,  a bijective proof and an enumerative  proof.

\begin{thm}\label{enum-down-up-words}
 For any $k \ge 2$ and $\ell \ge 2$, the number $N_{k,\ell}$ of down-up words over $[k]$ of length $\ell$ is equal to  the number of order ideals of  $Z_\ell \times (\mathbf{k-2})$.
\end{thm}

\noindent \textbf{Bijective Proof.}
Let $\mathcal{W}_{k,\ell}$ denote the set of down-up words over $[k]$ of length $\ell$. We shall build a bijection between $\mathcal{W}_{k,\ell}$ and ${J}(Z_\ell \times (\mathbf{k-2}))$.

 We first define a map $\Phi:\mathcal{W}_{k,\ell}\rightarrow {J}(Z_\ell \times (\mathbf{k-2}))$.  Given a down-up word $w=w_1 w_2 \cdots w_{\ell}$, we define the word $\alpha=\alpha_1 \alpha_2 \cdots \alpha_{\ell}$ as follows:
 $$\alpha_i=\left\{
\begin{array}{ll}
w_i-2,& \mbox{ if } i \mbox{ is odd},\\[3mm]
w_i-1,& \mbox{ if } i \mbox{ is even},
\end{array}\right.
$$
where $1\leq i\leq \ell$.
Then let $$\Phi(w)=\{(i,\beta_j)  : 1\leq i\leq \ell , \ 1 \le \beta_j \le \alpha_i\}.$$
For example, let $k=4$ and $\ell=7$, and consider the word $w=3241423$. Then,  $\alpha=1120211$ and thus $\Phi(w)=\{(1,1),(2,1),(3,1),(3,2),(5,1),(5,2),(6,1),
(7,1)\}$, which is an order ideal of $Z_7 \times \mathbf{2}$.

We need to show that this map is well defined.
It suffices to prove that $\Phi(w)$ is an order ideal of $Z_\ell \times (\mathbf{k-2})$, that is to say that,
if $(i',j')\leq (i,j)$ and $(i,j)\in \Phi(w)$ then $(i',j')\in \Phi(w)$.
From the definition of the order of $Z_\ell \times (\mathbf{k-2})$, we have that $i'\le_{Z_{\ell}} i$ and $j'\le j$.
Now, we divide the situation into two cases: $i'=i$ and $i'<_{Z_\ell}i$.
For the case $i'=i$, the argument is obviously true from the construction of $\Phi(w)$. We just need to consider the case $i'<_{Z_\ell}i$.
At this time, $i$ must be even,  and $i'$ can only be $i-1$ or $i+1$.
Since $(i,j)\in \Phi(w)$, we have that $\alpha_i\geq j$ and thus $w_i\geq j+1$.
From the fact that $w$ is a down-up word, it follows that  $w_{i'}> w_i$. Hence, $w_{i'} \ge j+2$ and thus  $\alpha_{i'}\geq j$.
From the construction of $\Phi(w)$, we obtain that  $(i',j')\in \Phi(w)$ for all $j'\leq j$, as desired.

 Next, we define a map $\Psi: {J}(Z_\ell \times (\mathbf{k-2})) \rightarrow \mathcal{W}_{k,\ell}$.
 Given an order ideal $I$ of $Z_\ell \times (\mathbf{k-2})$,
 we define a word  $\gamma=\gamma_1 \gamma_2 \cdots \gamma_{\ell}$ as follows. For each $1\leq i\leq\ell$, if there exists at least one $j$ such that $(i,j)\in I$, then let $\gamma_i$ be the maximum $j$. Otherwise, we let $\gamma_i$=0.
 The corresponding word $\Psi(I)$ is defined as $(2+\gamma_1)(1+\gamma_2)(2+\gamma_3)(1+\gamma_4)\cdots$.
  For exmaple, if $I=\{(1,1),(2,1),(3,1),(3,2),(5,1),(5,2),(6,1),
(7,1)\}$,
 then $\gamma=1120211$ and thus $w=3241423$.

 It is easy to see that, for any even integer $i$, we have $\gamma_i\leq \gamma_{i+1}$ and $\gamma_i\leq \gamma_{i-1}$, since $I$ is an order ideal. From the construction of $\Psi(I)$, we see that
 it is a down-up word.

Finally, it is not difficult to  prove that $\Psi \circ \Phi = id$ and $\Phi \circ \Psi = id$. Hence $\Phi$ is a bijection. This completes our bijective proof.\\

\noindent{\bf Enumerative  Proof.}
We first prove that the numbers in question satisfy the following recurrence relation, for $k\ge 3$ and $\ell\ge 2$,
\begin{align}\label{eq-N-rec}
 N_{k,\ell} = N_{k-1,\ell} + \sum_{i=0}^{\lfloor \frac{\ell-1}{2}\rfloor} N_{k-1,2i} N_{k,\ell-2i-1} - \delta_{\ell \text{ is even}} N_{k-1,\ell-2},
\end{align}
with the initial conditions
$N_{k, 0}=1$, $N_{k, 1}=k$ for $k\ge 2$, and $N_{2,\ell}=1$ for $\ell \ge 2$. To this end, we note that any down-up word $w$  over $[k]$ of length $\ell$ belongs to one of the following two cases.\\

\noindent
{\bf Case 1:} $w$ does not contain the letter $k$. Then the number we count is that of down-up words  over the alphabet $[k-1]$ of length $\ell$, which is $N_{k-1,\ell}$. This corresponds to the first term on the righthand side of \eqref{eq-N-rec}.\\

\noindent
{\bf Case 2:} $w$ is of the form $w_1kw_2$, where $w_1$ is a down-up word of even length over $[k-1]$, and $w_2$ is an up-down word over $[k]$. Note that the number of up-down words equals to  that of down-up words, as mentioned above. This corresponds to the second term on the right hand side of \eqref{eq-N-rec}.
The only exception occurs when the subword after the leftmost letter $k$ is of length one. It can be any letter in $[k-1]$, but $N_{k, 1}=k$. So, an additional term occurs, which fixes this.  In these cases, $\ell-2i-1$ equals 1, which means that $\ell$ is even. This completes the proof of \eqref{eq-N-rec}.

Now, let us denote the number of order ideals of  $Z_\ell \times \mathbf{k}$ by $M_{k,\ell}$.
We note that
Berman and K\"ohler \cite[Example 2.3]
{Berman1976Cardinalities} studied a similar recurrence for $M_{k,\ell}$, which is,
for $k\ge 1$ and $\ell\ge 1$,
\begin{align*}
 M_{k,\ell} = M_{k-1,\ell} + \sum_{i=0}^{\lfloor \frac{\ell-1}{2}\rfloor} M_{k-1,2i} M_{k,\ell-2i-1},
\end{align*}
with the initial conditions
$M_{k, 0}=1$  for $k\ge0$ and $M_{0,\ell}=1$ for $\ell\ge1$.

Owing to their akin recurrence relations, we made a minor change to the number $N_{k,\ell}$ to complete the proof.
We let $\widetilde{N}_{k,\ell}$ be $N_{k,\ell}$ except $\widetilde{N}_{k,1}=k-1$.  One can easily check that,
for $k\ge 3$ and $\ell\ge 1$,
\begin{align*}
 \widetilde{N}_{k,\ell} = \widetilde{N}_{k-1,\ell} + \sum_{i=0}^{\lfloor \frac{\ell-1}{2}\rfloor} \widetilde{N}_{k-1,2i} \widetilde{N}_{k,\ell-2i-1} ,
\end{align*}
with the initial conditions
$\widetilde{N}_{k, 0}=1$ for $k\ge2$ and $\widetilde{N}_{2,\ell}=1$ for $\ell\ge1$.
It follows immediately that, for $k\ge2$ and $\ell\ge0$, $$\widetilde{N}_{k, \ell} = M_{k-2,\ell},$$ since they have the same initial conditions and recurrence relations.
Together with the fact $N_{k, \ell} =\widetilde{N}_{k, \ell}$ except $\ell=1$, we obtain that
$$N_{k, \ell}=M_{k-2,\ell}$$
for $k\ge2$ and $\ell\ge2$. This completes our enumerative proof.
\qed

As an immediate corollary of Theorem~\ref{enum-down-up-words}, we have the following statement. 

\begin{thm} For $k\geq 3$ and $\ell\geq 2$, the numbers $N_{k,\ell}$ of down-up (equivalently, up-down) words of length $\ell$ over $[k]$ satisfy (\ref{eq-N-rec}) with the initial conditions $N_{k,0}=1$, $N_{k,1}=k$ for $k\geq 2$, and $N_{2,\ell}=1$ for $\ell\geq 2$. \end{thm} 

Note that the Fibonacci numbers have the following recurrence relations \cite[pp. 5--6]{Vorobiev2002Fibonacci}:
$$F_{2n}=\sum_{i=0}^{n-1}F_{2i+1}, \quad F_{2n+1}=1+\sum_{i=1}^{n}F_{2i}.$$
Using \eqref{eq-N-rec} and the fact that $N_{2,\ell}=1$ for $\ell\geq 2$, one can prove the following statement.

\begin{thm} For $\ell\geq 2$, $N_{3,\ell}=F_{\ell+2}$, the $(\ell+2)$th Fibonacci number. \end{thm}

\section{Enumeration of 123-avoiding up-down words}\label{123-up-down-sec}

In this section, we consider the enumeration of 123-avoiding up-down words. Denote $A_{k,\ell}$ the number of 123-avoiding up-down words  of length $\ell$ over the alphabet $[k]$, and $A_{k, \ell}^j$ the number of those words counted by $A_{k,\ell}$ that end with $j$.

\subsection{Explicit enumeration}

It is easy to see that
\begin{align}
 A_{k,2i}=\sum_{j=2}^{k} A_{k,2i}^j.
\end{align}

Next, we deal with the enumeration of $A^j_{k,2i}$. In what follows, for a word $w$, we have $\{w\}^*=\{\epsilon, w,ww,www,\ldots\}$, where $\epsilon$ is the empty word, and $\{w\}^+=\{w,ww,www,\ldots\}$.

\begin{lem}
For $k\ge 3$ and $2\le j\le k$, the numbers  $A_{k,2i}^{j}$ satisfy the following recurrence relation,
\begin{align}\label{rec}
A_{k,2i}^{j}
=  \sum_{i'=1}^{i} \left( A_{k-1, 2i'}^{j-1} - A_{k-1, 2i'}^{j}   + A_{k,2i'}^{j+1} \right),
\end{align}
with the boundary condition $A_{k,2i}^{k}= \binom{i+k-2}{i}$.
Furthermore, an explicit formula for $A^j_{k,2i}$ is
\begin{align}\label{A^j}
A^j_{k,2i}=\frac{j-1}{k-1} \binom{i+k-2}{i}  \binom{i+k-j-1}{i-1}.
\end{align}
\end{lem}

\begin{proof}
We first check  the boundary condition.
 When  $j=k$, the words must be of the form $$\{(k-1)k\}^{*}\{(k-2)k\}^{*}\ldots\{2k\}^{*}\{1k\}^{*}.$$
 The structure is dictated by the presence of the rightmost $k$; violating the structure, we will be forced to have an occurrence of the pattern $123$.
 Therefore, $A_{k,2i}^{k}=\binom{i+k-2}{i}$, where we applied the well known formula for the number of solutions of the equation $x_1+\cdots +x_{k-1}=i$ with $x_i\geq 0$ for $1\leq i\leq k-1$.

Now we proceed to deduce the recurrence relation \eqref{rec}.
All the legal words of length $2i$ ending with $j$ can be divided into the following cases according to the occurrence of  the letter $1$:

\noindent
{\bf Case 1:} For the legal words that contain the letter $1$, the letter $1$ must appear in the second last position, since otherwise it would lead to a $123$ pattern. We now divide all the legal words ending with $1j$ into the following subcases:
\\[4pt]
 \indent {\bf   Case 1.1:}  There is only one  word of the form $\{1j\}^{+}$.\\[4pt]
 \indent {\bf   Case 1.2:}  We deal with the words of the form $w\{j'j\}^{+}\{1j\}^{+}$, where $w$ is a legal word and $2\le j'\le j-1$. Note that $w$ cannot contain 1 because of an occurrence of $j'j$. Thus, we consider a legal word over the alphabet set $[2,k]$ of even length  which ends with $j'j$. By subtracting 1 from each letter of this word, we obtain a legal word over $[k-1]$ ending with $(j'-1)(j-1)$. Thus, the number of all words in this case is equal to that of all  words over $[k-1]$ ending with $j-1$, which is $\sum_{i'=1}^{i-1} A_{k-1, 2i'}^{j-1}$. \\[4pt]
 \indent {\bf   Case 1.3:} The others  are  the words of the form  $wj'^{+}\{1j\}^{+}$, where $w$ is a legal word and $j'\ge j+1$. Clearly, the number of such words in this case is $\sum_{i'=1}^{i-1}\sum_{j'=j+1}^{k}A_{k,2i'}^{j'}$. \\[6pt]
\noindent
{\bf Case 2:} We next deal with the legal words ending with $j$ over the alphabet $[k]\setminus \{1\}=\{2,3,\ldots,k\}$.
In this case, it has the same enumeration as that of the legal words over $[k-1]$ ending with $j-1$. The number of such words is $A_{k-1,2i}^{j-1}$.

Thus,  we have the following recurrence relation
\begin{align}\label{formula-A-k-2i-j}
A_{k,2i}^{j}
=  1+\sum_{i'=1}^{i} A_{k-1, 2i'}^{j-1}  + \sum_{i'=1}^{i-1}\sum_{j'=j+1}^{k}A_{k,2i'}^{j'}.
\end{align}
From (\ref{formula-A-k-2i-j}), we have that
\begin{align*}
A_{k,2i}^{j} - A_{k,2i}^{j+1}
=  \sum_{i'=1}^{i} \left( A_{k-1, 2i'}^{j-1}-A_{k-1, 2i'}^{j} \right)  + \sum_{i'=1}^{i-1}A_{k,2i'}^{j+1},
\end{align*}
and therefore the recurrence  \eqref{rec} follows.

Now we deduce the formula \eqref{A^j} for $A_{k,2i}^{j}$.
Let $$A'(k,i,j) = \frac{j-1}{k-1} \binom{i+k-2}{i} \binom{i+k-j-1}{i-1}.$$
We next prove that $A_{k,2i}^{j} = A'(k,i,j)$ by induction on $k-j$ and $k$. We shall show that these numbers  have the same base case and satisfy the same recursion.
Indeed, for $k=j\ge 2$, this fact is obviously true, since $A'(k,i,k)=A^k_{k,2i}$.
We will now check that $A'(k,i,j)$ satisfy the following recurrence relation:
\begin{align}\label{rec-formula-Bs}
A'(k,i,j)
=  \sum_{i'=1}^{i}  \left(  A'(k-1, i',j-1) - A'(k-1, i',j)   +A'(k,i',j+1)  \right).
\end{align}
Indeed, (\ref{rec-formula-Bs}) is true if and only if
\begin{align*}
A'(k,i,j)-A'(k,i-1,j)
=   A'(k-1, i, j-1) - A'(k-1, i, j)   +A'(k, i, j+1),
\end{align*}
while the later equation is easy to check to be true.
This completes the proof.
\end{proof}

Further, the number of 123-avoiding up-down words  of length $2i$ over $[k]$ is
 $$A_{k,2i}=\sum_{j=2}^k\frac{j-1}{k-1} \binom{i+k-2}{i} \binom{i+k-j-1}{i-1}=\frac{1}{i+1} \binom{i+k-2}{i}   \binom{i+k-1}{i}.$$
The last equality can be deduced from the Gosper algorithm \cite{Petkovsek1996AB}.

Now we consider legal words of odd length. For any legal word of length $2i\ (i\ge 1)$ ending with $j\ (2\le j\le k)$, we can adjoin any letter in $[j-1]$ at the end to form an up-down word  of length $2i+1$ over  $[k]$. In fact, such words are necessarily 123-avoiding.
So, we obtain that
\begin{align*}
A_{k,2i+1}= & \sum_{j=2}^k(j-1)A^j_{k,2i}\\
=& \sum_{j=2}^k\frac{(j-1)^2}{k-1} \binom{i+k-2}{i}  \binom{i+k-j-1}{i-1}\\
= & \frac{i+2k-2}{(i+1)(i+2)}\binom{i+k-2}{i}  \binom{i+k-1}{i}.
\end{align*}
Also, the last equality can be deduced from the Gosper algorithm \cite{Petkovsek1996AB}.

Hence, we have proved the following theorem.

\begin{thm}\label{main-thm} For $A_{k,\ell}$, the number of $123$-avoiding up-down words of length $\ell$ over $[k]$, $A_{k,0}=1$, $A_{k,1}=k$, and for $\ell\geq2$,
\begin{align}\label{main-formula-A-k-l}
 A_{k,\ell}=
\begin{cases}
  \frac{1}{i+1} \binom{i+k-2}{i} \binom{i+k-1}{i}, & \mbox{ if } \ell=2i,\\[6pt]
  \frac{i+2k-2}{(i+1)(i+2)}  \binom{i+k-2}{i}  \binom{i+k-1}{i}, & \mbox{ if } \ell=2i+1.
\end{cases}
\end{align}
\end{thm}

\subsection{Generating functions}

%
%

In this subsection, an expression for the generating function  for the numbers   $A_{k,i}$  of  123-avoiding up-down words of length $i$ over $[k]$ is given.
We adopt the notation of {\em Narayana polynomials}, which are defined as $N_0(x)=1$ and,  for $n\geq 1$,
$$N_{n}(x)=\sum_{i=0}^{n-1}\frac{1}{i+1} \binom{n}{i}\binom{n-1}{i}x^i.$$
Due to Brenti \cite{Brenti1989Unimodal} and  Reiner and Welker \cite[Section 5.2]{Reiner2005Charney}, a remarking generating function for $A_{k,2i}$ can be expressed as follows: 
\begin{align}\label{eq:Narayana}
\sum_{i\geq 0}A_{k,2i}x^i=\frac{N_{k-2}(x)}{(1-x)^{2k-3}}.
\end{align}
On the other hand, by Theorem \ref{main-thm}, a routine computation leads to the following identity,
\begin{align}\label{eq:123-odd}
  A_{k,2i-1}&=A_{k,2i}-A_{k-1,2i},
\end{align}
for all $i\geq 2$.
(Note that we shall also give a combinatorial interpretation of \eqref{eq:123-odd} in Section \ref{all-cases}.)
Thus, together with $A_{k,1}=k$, it follows that
\begin{align*}
 \sum_{i\geq 1}A_{k,2i-1}x^i&= x+ \sum_{i\geq 1}A_{k,2i}x^i-\sum_{i\geq 1}A_{k-1,2i}x^i\\
 &=x+\frac{N_{k-2}(x)}{(1-x)^{2k-3}}-\frac{N_{k-3}(x)}
 {(1-x)^{2k-5}}\\
 &=x+\frac{N_{k-2}(x)-(1-x)
 ^2N_{k-3}(x)}{(1-x)^{2k-3}}.
\end{align*}
Hence, we are ready to obtain the main result of this subsection,
\begin{align*}
 \sum_{i\geq 0}A_{k,i}x^i&=\sum_{i\geq0}A_{k,2i}x^{2i}
 + \sum_{i\geq 1}A_{k,2i-1}x^{2i-1}\\
 &=\frac{N_{k-2}(x^2)}{(1-x^2)^{2k-3}}+x+ \frac{N_{k-2}(x^2)-
 (1-x^2)^2N_{k-3}(x^2)}{x(1-x^2)^{2k-3}}\\
 &=x+\frac{(1+x)N_{k-2}(x^2)-(1-x^2)^2N_{k-3}(x^2)}{x(1-x^2)^{2k-3}}.
\end{align*}

\section{A bijection between  \texorpdfstring{$S^{132}_{k,2i}$}{Lg}  and \texorpdfstring{$S^{123}_{k,2i}$}{Lg}}\label{bijection-sec}

Let $p$ be a pattern and $S^{p}_{k,\ell}$ be the set of  $p$-avoiding up-down words of length $\ell$ over $[k]$.
In this section, we will build a bijection between $S^{132}_{k,2i}$  and $S^{123}_{k,2i}$.

The idea here is to introduce the notion of irreducible words and show that irreducible words in $S^{132}_{k,2i}$ can be mapped in a 1-to-1 way into irreducible words in $S^{123}_{k,2i}$, while reducible words in these sets can be mapped to each other as well.

\begin{definition} A word $w$ is {\em reducible}, if $w=w_1w_2$ for some non-empty words $w_1$ and $w_2$, and  each letter in $w_1$ is no less than every letter in $w_2$. The place between $w_1$ and $w_2$ in $w$ is called a {\em cut-place}.\end{definition}

For example, the word 242313 is irreducible, while the word 341312 is reducible (it can be cut into 34 and 1312).

Note that in a reducible up-down word, if we have a cut-place, and there are equal elements on both sides of it, then to the left such elements must be bottom elements, and to the right they must be top elements.

\begin{lem} A word $w$ in $S^{132}_{k,2i}$ is irreducible if and only if  $w=w_1xy$, where $w_1$ is a word in $S^{132}_{k,2i-2}$, $x$ is the minimum letter in $w$ (possibly, there are other copies of $x$ in $w$) and $y$ is the maximum letter in $w$ (possibly, there are other copies of $y$ in $w$).\end{lem}

 \begin{proof}
  If $x$ is not the minimum element in $w$, then the element right before it, and the minimum element in $w$ will form the pattern 132.
  Since $w$ is irreducible, $y$ is forced to be no less than the minimum element in $w_1$.
  On the other hand, if $y$ is not the maximum element in $w$, then  the maximum one  in $w_1$ and the element just preceding it will form the pattern 132.
  This completes the proof.
 \end{proof}

Now, given a word $w$ in $S^{132}_{k,2i}$, we can count  in how many ways it can be extended to an irreducible word in $S^{132}_{k,2i+2}$. Suppose that $a$ and $b$ are the minimum and the maximum elements in $w$, respectively. Then the number of extensions of $w$ in $S^{132}_{k,2i+2}$ is $a\cdot(k-b+1)$, since there are $a$ choices of the next to last element and $k-b+1$ choices of the last element.

Next, we discuss a procedure of turning any word $w$ in $S^{123}_{k,2i}$ into an irreducible word in $S^{123}_{k,2i+2}$. From this procedure, it would be clear that the number of choices  is $a\cdot(k-b+1)$, where $a$ and $b$ are the minimum and the maximum elements in $w$, respectively.

Suppose that $w=b_1 t_1 b_2 t_2 \cdots b_i t_i$, where $b_j$'s and $t_j$'s stand for bottom and top elements, respectively.
To obtain the desired word, we inserting a new top element $x$, shift the bottom elements one position to the left, and then insert one more bottom element $y$.
Then, the extension is of the form
$$w'=b_1 x b_2 t_1 b_3 t_2 \cdots b_i t_{i-1} y t_i,$$
where $x\ge b$ and $y\le a$, and again, $a$ and $b$ are the minimum and the maximum elements in $w$, respectively.
For example, if  $w = 242313\in S^{123}_{5,6}$, $a=1$ and $b=4$,  then $w'$ can be $24241313$ or $25241313$.

To see that  the  resulting word $w'$ is an up-down word. In fact,  it is sufficient to show that $b_{j+1} < t_{j}$ for $1\le j\le i-1$ and $t_j > b_{j+2}$ for $1\le j\le i-2$. The first inequality follows from the fact that $w$ is an up-down word, while the second one is true, because otherwise $t_j\le b_{j+2}<t_{j+1}$ and thus $b_j t_j t_{j+1}$  would form a $123$ pattern. 
We also claim that $w'$ belongs to  $S^{123}_{k,2i+2}$. 
An equivalent condition an up-down  $w$ is $123$ avoiding is that $w$ satisfy  
$b_1\geq b_2\geq\dots \geq b_i$  and  $t_1\geq t_2\geq\dots \geq t_i$. 
From the construction of $w'$, we obtain that  $w'$ is also $123$ avoiding.
Besides,  $w'$ is irreducible, since $b_j < t_j$ for $1\le j\le i$.

Now, a bijection between $S^{132}_{k,2i}$  and $S^{123}_{k,2i}$ is straightforward to set recursively, with a trivial base case of words of length 2 mapped to themselves. Indeed, if we assume that we can map all words in $S^{132}_{k,2i}$ to all words in $S^{123}_{k,2i}$, then applying the same choice of $x$ and $y$, we can map all irreducible words in $S^{132}_{k,2i+2}$ to all irreducible words in $S^{123}_{k,2i+2}$. Finally, each reducible word $w$ is of the form $w_1w_2$, where $w_1$ is irreducible word with maximum possible even length. But then $w_1$ and $w_2$ are of smaller lengths than $w$, and we can map them recursively.

\begin{table}[!htbp]
 \renewcommand{\arraystretch}{1.3}
\begin{center}
\begin{tabular}{|ccc|lll|}
\hline
$S^{132}_{4,4}$ & & $S^{123}_{4,4}$ &$S^{132}_{4,4}$ & & $S^{123}_{4,4}$ \\
\hline
1212&& 1212 &2312&& 2312\\
2412 && 2412&3412 && 3412\\
1213&& 1312&1313 &&1313\\
2313&& 2313&3413& &3413\\
2323&& 2323&3423&& 3423\\
1214&& 1412&1314 &&1413\\
2314& &2413&1414 &&1414\\
2414&& 2414&3414& &3414\\
2324 &&2423&2424& &2424\\
3424&& 3424&3434&& 3434\\
\hline
\end{tabular}
\end{center}
\caption{The bijection $\phi : S^{132}_{4,4} \rightarrow S^{123}_{4,4}$.}\label{tab2}
\end{table}

For example, $w=3435121213\in S^{132}_{5,10}$ is reducible, since it can be cut into $3435$ and $121213$. We calculate that $\phi(3435)$ by two steps: first $\phi(34)=34$ and the second $\phi(3435)=3534$. Similarly, $\phi(121213)=131212$. It follows that $\phi(3435121213)=3534131212$, i.e. the word $3435121213$ in $S^{132}_{5,10}$ is mapped to $3534131212$ in $S^{123}_{5,10}$. Also, see Table \ref{tab2}  showing images of all words in $S^{132}_{4,4}$. 

\section{Enumeration of 132-avoiding up-down words}\label{sec-132-av}
In this section, we consider the enumeration of 132-avoiding up-down words.
Denote $B_{k,\ell}$ the number of 132-avoiding up-down words of length $\ell$ over the alphabet $[k]$, and $B_{k, \ell}^{j}$ the number of those words counted by $B_{k,\ell}$ whose letter in the second last position is $j$.

From the bijection, it follows that
\begin{align}
 B_{k,2i} = A_{k,2i}=\frac{1}{i+1} \binom{i+k-2}{i}   \binom{i+k-1}{i}.
\end{align}

For any legal word of length $2i\ (i\ge 1)$ whose letter in the second last position is $j\ (1\le j\le k-1)$,  the minimum letter in this word is also $j$, since it is 132-avoiding. By subtracting $j-1$ from each letter of this word, we obtain a legal word over $[k-j+1]$  whose letter in the second last position is $1$.  Thus, we have that
$$B_{k,2i}^{j} = B_{k-j+1,2i}^{1}.$$
Similarly, we obtain that
\begin{align*}
 B_{k,2i}^{1} = B_{k,2i}-\sum_{j=2}^{k-1}B_{k,2i}^{j}=B_{k,2i}-B_{k-1,2i}= \frac{i+2 k-3 }{(i+1)(k-1)} \binom{i+k-3}{i-1} \binom{i+k-2}{i}.
\end{align*}

Now, we consider legal words of odd length. For any legal word of length $2i\ (i\ge 1)$ whose letter in the second last position is $j\ (1\le j\le k-1)$,  the minimum letter in the word is also $j$ and hence, we can adjoin any letter in $[j]$ at the end to form an up-down word on  $[k]$ of length $2i+1$. In fact, such words are necessarily 132-avoiding.
So, we obtain that for $i\ge 1$
\begin{align*}
B_{k,2i+1}= & \sum_{j=1}^{k-1} j B_{k,2i}^{j}\\
=& \sum_{j=1}^{k-1} j B_{k-j+1,2i}^{1}\\
=& \sum _{j=1}^{k-1}j\, \frac{i+2 (k-j)-1}{(i+1)(k-j)}  \binom{i+k-j-2}{i-1} \binom{i+k-j-1}{i}\\
=& \sum _{j=1}^{k-1} \frac{(k-j)  (i+2 j-1)}{ j(i+1)} \binom{i+j-2}{i-1} \binom{i+j-1}{i}.
\end{align*}

Unfortunately, we were not able to find a closed form formula for $B_{k,2i+1}$. We conclude this section with listing expressions for $B_{k,2i+1}$ for $k=3,4,5,6$ and $i\ge 1$:
\begin{align*}
B_{3,2i+1} & = \frac{1}{2} (i^2+3 i+4), \\
B_{4,2i+1} & = \frac{1}{12} (i^4+8 i^3+29 i^2+46 i+36), \\
B_{5,2i+1} & = \frac{1}{144} (i^6+15 i^5+103 i^4+381 i^3+832 i^2+972 i+576), \\
B_{6,2i+1} & = \frac{1}{2880} (i^8+24 i^7+266 i^6+1704 i^5+6929 i^4+18096 i^3+30244 i^2+29136 i+14400).
\end{align*}

\section{Enumeration of 312-avoiding up-down words}\label{312-up-down-sec}
In this section, we consider the enumeration of 312-avoiding up-down words. 
Denote $C_{k,\ell}$ the number of 312-avoiding up-down words of length $\ell$ over the alphabet $[k]$, and $C_{k, \ell}^j$ the number of those words counted by $C_{k,\ell}$ that end with $j$.

Recall the definition of the complement $w^c$ of a word $w$ given in Section~\ref{intro}. Also, for a word $w=w_1w_2\cdots w_n$, its {\em reverse word} $w^{r}$ is given by $w^{r}=w_{\ell}w_{\ell-1}\cdots w_{1}$. 
It is clear that the operations of reverse and complement are both bijections on alternating words.

Let $p$ be a pattern and let $\widehat{S}^p_{k,\ell}$ denote the set of all $p$-avoiding down-up words of length $\ell$ over an alphabet $[k]$.

We first consider the words of odd length.

\begin{prop}For all $k\geq 1$ and $i\geq 0$,
we have that the number of $312$-avoiding up-down words of length $2i+1$ on $[k]$  is the same as that of $123$-avoiding up-down words on $[k]$ of the same length. Namely,
$$C_{k,2i+1}=A_{k,2i+1}.$$
\end{prop}

\begin{proof}
 We shall prove this theorem by  establishing  a bijection between $S^{312}_{k,2i+1}$ and $S^{123}_{k,2i+1}$.
Applying the complement operation to the former of these sets, and reverse and complement to the latter set, it suffices to show that there exists a bijection  between $\widehat{S}^{132}_{k,2i+1}$ and $\widehat{S}^{123}_{k,2i+1}$. 

The map $\psi(w): \widehat{S}^{132}_{k,2i+1} \rightarrow \widehat{S}^{123}_{k,2i+1}$ is defined as follows.
For any $w=w_1w_2\cdots w_{2i+1} \in \widehat{S}^{132}_{k,2i+1}$, let $w'$ be $w_2\cdots w_{2i+1}$.  It is clear that $w'\in S^{132}_{k,2i}$. Thus let $$\psi(w)=w_1\phi(w'),$$ where the  map $\phi: S^{132}_{k,2i} \rightarrow S^{123}_{k,2i}$ is described in Section \ref{bijection-sec}. 

We need to show that $\psi$ is well-defined.
From the construction of $\phi$, we see that $\phi$ preserves the first letter, i.e.  $\phi(w')_1=w'_1=w_2$. Therefore,  it follows that $\psi(w)\in \widehat{S}^{123}_{k,2i+1}$. Hence, the map $\psi$ is well-defined. 

It is not difficult to see by construction that $\psi$ is a bijection.
Hence, we get  a bijection between $\widehat{S}^{132}_{k,2i+1}$ and $\widehat{S}^{123}_{k,2i+1}$.
This completes the proof.
\end{proof}

Now let us consider the words of even length. Note that
$$C_{k,2i}=\sum_{j=2}^{k} C_{k,2i}^j.$$ 

For any word $w \in S^{312}_{k,2i}$ whose last letter is $j$, the maximum letter of $w$ is also $j$, since the word is $312$-avoiding and having $j'>j$ in $w$ would lead to an occurrence of three letters $j'w_{2i-1}j$ forming the pattern $312$. Thus, we have that $$C_{k,2i}^j=C_{j,2i}^j,$$ where $2\leq j\leq k$ .

Moreover, for any word in $S^{312}_{j,2i}$ ending with $j$, we can remove $j$ to form a word of length $2i-1$, which is also $312$-avoiding. On the other hand, for any word  $S^{312}_{j,2i-1}$, we can adjoin a letter $j$ at the end to form a $312$-avoiding word of length $2i$. Thus, $$C_{j,2i}^j=C_{j,2i-1}.$$
So,  we obtain that
\begin{align*}
 C_{k,2i} &= \sum_{j=2}^{k}C_{j,2i-1}\\
 &=\sum_{j=2}^{k}\frac{i-3+2j}{i(i+1)}\binom{i+j-3}{ i-1}\binom{i+j-2}{i-1}\\
 &=\frac{1}{i+1} \binom{i+k-2}{i}\binom{i+k-1}{i}\\
 &=A_{k,2i},
\end{align*}
where  the  second last equality can be deduced from the Gosper algorithm \cite{Petkovsek1996AB}.

We have just obtained the main result in this section.
\begin{thm}\label{123-312}
The sets of $312$-avoiding up-down words and $123$-avoiding up-down words are equinumerous, that is, $$C_{k,\ell}=A_{k,\ell}$$ for all $k\geq 1$ and $\ell\geq 0.$
\end{thm}

\section{Enumeration of other pattern-avoiding up-down words}\label{all-cases}

In this section, we consider the enumeration of other pattern-avoiding up-down words.
In order to avoid confusion, let $N_{k,\ell}(p)$ denote the number of $p$-avoiding up-down words of length $\ell$ over the alphabet $[k]$.

We first focus on all six length 3 permutation patterns to be avoided on up-down words of odd length.
\begin{thm} For all $k\geq 1$ and $i\geq 0$, we have
$$N_{k,2i+1}(123)=N_{k,2i+1}(312)=N_{k,2i+1}(213)
=N_{k,2i+1}(321)$$
and
$$N_{k,2i+1}(132)=N_{k,2i+1}(231).$$
\end{thm}

\begin{proof}
 Through the reverse operation, the following equations hold:
$$N_{k,2i+1}(132)=N_{k,2i+1}(231),$$
$$N_{k,2i+1}(123)=N_{k,2i+1}(321),$$ and
$$N_{k,2i+1}(312)=N_{k,2i+1}(213).$$
Combing  with Theorem \ref{123-312}, the proof is complete.
\end{proof}

Next,  we obtain the following result for the case of the even length.
\begin{thm}
For all $k\geq 1$ and $i\geq 1$, there is 
$$N_{k,2i}(123)=N_{k,2i}(132)=N_{k,2i}(312)=N_{k,2i}(213)
=N_{k,2i}(231).$$
\end{thm}
\begin{proof}
 Through the complement and reverse operations, it follows that 
$$N_{k,2i}(132)=N_{k,2i}(213),$$ and
$$N_{k,2i}(231)=N_{k,2i}(312).$$
From Section \ref{bijection-sec}, we have that
$$N_{k,2i}(132)=N_{k,2i}(123).$$
Together with Theorem \ref{123-312}, we complete the proof.
\end{proof}

In the rest of this section, we deal with the only remaining case,  321-avoiding up-down words of even length.  Our approach is based on deriving the desired from an alternative enumeration of  123-avoiding up-down words.

All 123-avoiding up-down  words of length $\ell$ over $[k]$, for $\ell \geq 4$, can be divided into the following two cases:
\begin{itemize}
\item Legal words containing no $k$ in them. These words are counted by $A_{k-1, \ell}$.
\item Legal words that contain at least one $k$. Such words $w=w_1w_2\cdots w_{\ell}$ are necessarily of the form $w_1kw_3\cdots w_{\ell}$, since otherwise  $w_1w_2k$ would be an occurrence of the pattern $123$. 
Clearly,  $w_1\geq w_3$ (otherwise, $w_1w_3w_4$ would form the pattern 123).
We let $w'$ be $k w_3\cdots w_{\ell}$ if $w_1=w_3$ and  $w_1 w_3\cdots w_{\ell}$ if $w_1>w_3$.
Clearly, this is a reversible procedure and the obtained words $w'$ are 123-avoiding down-up words.  By applying the complement  operation, we obtain 321-avoiding up-down words over $[k]$ of length $\ell-1$.
\end{itemize}
It follows that for $\ell \geq 4$,  $$A_{k, \ell}=A_{k-1, \ell}+N_{k, \ell-1}(321),$$
and thus 
$$N_{k, \ell-1}(321)= A_{k, \ell}-A_{k-1, \ell}.$$
Hence, by Theorem \ref{main-thm},  we are ready to obtain an expression for  $N_{k, \ell}(321)$.

\begin{thm}\label{thm-321} For the number of $321$-avoiding up-down words of length $\ell$ over $[k]$, $N_{k,0}(321)=1$, $N_{k,1}(321)=k$,  $N_{k,2}(321)=\binom{k}{2}$,  and for $\ell\geq 3$,
\begin{align*}
 N_{k, \ell}(321)=
\begin{cases}
\frac{i (i+2 k-3) (i+2 k-2)+2 (k-2) (k-1)}{(i+1) (i+2) (k-2) (k-1)}  \binom{i+k-2}{i} \binom{i+k-3}{i} , & \mbox{ if } \ell=2i,\\[6pt]
  \frac{i+2k-2}{(i+1)(i+2)}  \binom{i+k-2}{i}  \binom{i+k-1}{i}, & \mbox{ if } \ell=2i+1.
\end{cases}
\end{align*}
\end{thm}

Since $N_{k,2i+1}(123)=N_{k,2i+1}(321)$, we actually  give another approach to deal with $321$-avoiding up-down words of odd length.

\section{Concluding remarks}\label{last-sec}

In this paper we initiated the study of (pattern-avoiding) alternating words. In particular, we have shown that 123-avoiding up-down words of even length are given by the Narayana numbers. Thus, alternating words can be used, for example, for encoding Dyck paths with a specified number of peaks \cite{GaoZhang}. To our surprise, the enumeration of 123-avoiding up-down words turned out to be easier than that of 132-avoiding up-down words, as opposed to similar studies for permutations, when the structure of 132-avoiding permutations is easier than that of 123-avoiding permutations.

Above, we gave a complete classification of avoidance of permutation patterns of length~$3$ on alternating words. We state it as an open direction of research to study avoidance of longer patterns and/or patterns of different types (see~\cite{Kitaev2011Patterns}) on alternating (up-down or down-up) words.

\vskip 3mm
\noindent {\bf Acknowledgments.} 
This work was supported by the 973 Project, the PCSIRT Project of the Ministry of Education and the National Science Foundation of China.

%
%
%
%
%
%
%
%
%
%
%
%

\end{document}